\documentclass[twoside,reqno]{amsart}
\usepackage{latexsym}%%triangles:\lhd\rhd\unlhd\unrhd%
\newcommand{\qdn}{\hspace*{-1.5mm}}
\newcommand{\qqdn}{\hspace*{-2.5mm}}
\newcommand{\xqdn}{\hspace*{-5.0mm}}
\newcommand{\xxqdn}{\hspace*{-10mm}}

\newcommand{\binm}{\binom}

\newcommand{\be}{\begin{equation}}
\newcommand{\ee}{\end{equation}}
\newcommand{\ba}{\begin{array}}
\newcommand{\ea}{\end{array}}
\newcommand{\bmn}{\begin{eqnarray}}
\newcommand{\emn}{\end{eqnarray}}
\newcommand{\bnm}{\begin{eqnarray*}}
\newcommand{\enm}{\end{eqnarray*}}
\newcommand{\bln}{\begin{subequations}}
\newcommand{\eln}{\end{subequations}}

\newtheorem{thm}{Theorem}%[section]

\newtheorem{corl}[thm]{Corollary}
\newtheorem{prop}[thm]{Proposition}

\newtheorem{entry}{Entry}%%%%%%%%%%%%%%%%

\newcommand{\bbtm}[4]{\bibitem{kn:#1}{#2,}~{#3,}~{#4.}}
\newcommand{\cito}[1]{\cite{kn:#1}}

\begin{document} %%%%%%%%%% This paper is published in %%%%%%%
{%%\fns \today\hfill\copyright%% Printed in China} %%%%%%%%%%%%%%%
%%%%%%%%%%%%%%%%%%%%%%%%%%%%%%%%%%%%%%%%%%%%%%%%%%%%%%%%%%%%%%
\title{A family of summation formulas \\involving generalized harmonic numbers}
\author{$^1$Chuanan Wei, $^2$Qinglun Yan, $^3$Dianxuan Gong}

\footnote{\emph{2010 Mathematics Subject Classification}: Primary
05A19 and Secondary 40A25}

\dedicatory{$^1$Department of Information Technology\\
 Hainan Medical College, Haikou 571101, China\\
   $^2$College of Mathematics and Physics\\
   Nanjing University of Posts and Telecommunications,
    Nanjing 210046, China\\
    $^3$College of Sciences\\
 Hebei Polytechnic University, Tangshan 063009, China}

 %Email address: \emph{}}
\thanks{\emph{Email addresses}:
weichuanan@yahoo.com.cn (C. Wei), yanqinglun@yahoo.com.cn (Q. Yan),
gongdianxuan@yahoo.com.cn (D. Gong)}

 \keywords{Telescoping method;
 Derivative operator; Generalized harmonic numbers}

\begin{abstract}
Combining the derivative operator with a binomial sum from the
telescoping method, we establish a family of summation formulas
involving generalized harmonic numbers.
\end{abstract}

%%%%%%%%%%%%%%%%%%%%%%%%%%%%%%%%%%%%%%%%%%%%%%%%%%%%%%%%%%%%%%%%%%%
\maketitle\thispagestyle{empty}%%%%%%%%%%%%%%%%%%%%%%%%%%%%%%%%%%%%
\markboth{C. Wei, Q. Yan, D. Gong}%%%%%%%%%%%%%%%%%%%%%%%%%%%%
         {A family of summation formulas involving generalized harmonic numbers}

%%%%%%%%%%%%%%%%%%%%%%%%%%%%%%%%%%%%%%%%%%%%%%%%%%%%%%%%%%%%%%%%%%%
%%%%%%%%%%%%%%%%%%%%%%%%%%%%%%%%%%%%%%%%%%%%%%%%%%%%%%%%%%%%%%%%%%%
%%%%%%%%%%%%%%%%%%%%%%%%%%%%%%%%%%%%%%%%%%%%%%%%%%%%%%%%%%%%%%%%%%%
\section{Introduction}
%%%%%%%%%%%%%%%%%%%%%%%%%%%%%%%%%%%%%%%%%%%%%%%%%%%%%%%%%%%%%%%%%%%
%%%%%%%%%%%%%%%%%%%%%%%%%%%%%%%%%%%%%%%%%%%%%%%%%%%%%%%%%%%%%%%%%%%
For $x\in \mathbb{C}$ and $\l, n\in \mathbb{N}_0$, define the
functions $H_{n}^{\langle \l\rangle}$(x) by
\[H_{0}^{\langle \l\rangle}(x)=0\quad\text{and}\quad
  H_{n}^{\langle \l\rangle}(x)
  =\sum_{k=1}^n\frac{1}{(x+k)^l}\quad\text{with}\quad n=1,2,\cdots.\]
Fixing $x=0$ in the functions just mentioned, we obtain the
generalized harmonic numbers:
\[H_{0}^{\langle l\rangle}=0\quad\text{and}\quad
  H_{n}^{\langle l\rangle}
  =\sum_{k=1}^n\frac{1}{k^l}\quad\text{with}\quad n=1,2,\cdots.\]
When $l=1$, they reduce to the classical harmonic numbers:
\[H_{0}=0\quad \text{and}\quad H_{n}
=\sum_{k=1}^n\frac{1}{k}\quad \text{with}\quad n=1,2,\cdots.\]
 There exist many elegant identities involving generalized harmonic numbers. They
can be found in the papers \cito{andrews}-\cito{zheng}.

For a differentiable function $f(x)$, define the derivative operator
$\mathcal{D}_x$ by
 \bnm
&&\xqdn\mathcal{D}_xf(x)=\frac{d}{dx}f(x).
 \enm

 Then it is not difficult to show the following two
 derivatives:
 \bnm
&&\mathcal{D}_x\binm{x+n}{n}=\binm{x+n}{n}H_{n}(x),\\[1mm]
&&\mathcal{D}_x\,H_{n}^{\langle \l\rangle}(x)=
 -lH_{n}^{\langle\l+1\rangle}(x).
  \enm

For a complex sequence $\{\tau_k\}_{k\in \mathbb{Z}}$, define
 the difference operator by
\[\nabla\tau_k=\tau_k-\tau_{k-1}.\]
Then we have the following relation:
 \bnm
\nabla\frac{\binm{y+k+1}{k}}{\binm{x+k}{k}}=\frac{\binm{y+k}{k}}{\binm{x+k}{k}}\frac{y-x+1}{y+1}.
 \enm
Combining the last equation and the telescoping method:
 \bnm
\qquad\sum_{k=1}^{n}\nabla\tau_k=\tau_n-\tau_{0},
 \enm
we get the simple binomial sum:
 \bmn \label{source-bino}
\sum_{k=1}^n\frac{\binm{y+k}{k}}{\binm{x+k}{k}}=\frac{\binm{y+n+1}{n}}{\binm{x+n}{n}}\frac{y+1}{y-x+1}-\frac{y+1}{y-x+1}.
 \emn

By means of the derivative operator $\mathcal{D}_x$ and the binomial
sum \eqref{source-bino}, we shall explore systematically closed
expressions for the family of sums:
\[\sum_{k=1}^nk^iH_k^{\langle \l\rangle}(x)\:\:\text{with}\:\: i,\l\in\mathbb{ N}_0.\]
When $x=p$ with $p\in\mathbb{ N}_0$, they give closed expressions
for the following sums:
\[\sum_{k=1}^nk^iH_{p+k}^{\langle \l\rangle}.\]
%%%%%%%%%%%%%%%%%%%%%%%%%%%%%%%%%%%%%%%%%%%%%%%%%%%%%%%%%%%%%%%%%%%
%%%%%%%%%%%%%%%%%%%%%%%%%%%%%%%%%%%%%%%%%%%%%%%%%%%%%%%%%%%%%%%%%%%
\section{Summation formulas}
%%%%%%%%%%%%%%%%%%%%%%%%%%%%%%%%%%%%%%%%%%%%%%%%%%%%%%%%%%%%%%%%%%%

\begin{thm}\label{thm-a}
For $x\in \mathbb{C}$ and $\l\in\mathbb{ N}_0$, there holds the
summation formula:
 \bnm
\sum_{k=1}^nH_k^{\langle \l+1\rangle}(x)=
(x+n+1)H_n^{\langle\l+1\rangle}(x)-H_n^{\langle \l\rangle}(x).
 \enm
\end{thm}

\begin{proof}

Applying the derivative operator $\mathcal{D}_x$ to
\eqref{source-bino}, we achieve the identity:
 \bmn \label{source}
\sum_{k=1}^n\frac{\binm{y+k}{k}}{\binm{x+k}{k}}H_k(x)
=\frac{y+1}{y-x+1}\frac{\binm{y+n+1}{n}}{\binm{x+n}{n}}
 \bigg\{H_n(x)-\frac{1}{y-x+1}\bigg\}+\frac{y+1}{(y-x+1)^2}.
 \emn
Letting $y=x$ in \eqref{source}, we attain the case $\l=0$ of
Theorem \ref{thm-a}:
 \bnm
\sum_{k=1}^nH_k(x)=(x+n+1)H_n(x)-n.
 \enm
Suppose that the following identity
 \bnm
  \qqdn\sum_{k=1}^nH_k^{\langle \l+1\rangle}(x)
   =(x+n+1)H_n^{\langle \l+1\rangle}(x)
   -H_n^{\langle \l\rangle}(x)
 \enm
 is true. Applying the derivative operator
$\mathcal{D}_x$ to the last equation, we have
 \bnm
\sum_{k=1}^nH_k^{\langle \l+2\rangle}(x)
  =(x+n+1)H_n^{\langle \l+2\rangle}(x)
 -H_n^{\langle \l+1\rangle}(x).
 \enm
This proves Theorem \ref{thm-a} inductively.
\end{proof}

Making $x=p$ in Theorem \ref{thm-a}, we get the following equation.

\begin{corl}
 For $\l,p\in\mathbb{ N}_0$, there holds the summation
formula:
 \bnm
\sum_{k=1}^n H_{p+k}^{\langle \l+1\rangle}
=(p+n+1)H_{p+n}^{\langle\l+1\rangle}-(p+1)H_p^{\langle\l+1\rangle}
-H_{p+n}^{\langle\l\rangle}+H_{p}^{\langle \l\rangle}.
 \enm
\end{corl}

%%%%%%%%%%%%%%%%%%%%%%%%%%%%%%%%%%%%%%%%%%%%%%%%%%%%%%%%%%%%%%%%%%%
%%%%%%%%%%%%%%%%%%%%%%%%%%%%%%%%%%%%%%%%%%%%%%%%%%%%%%%%%%%%%%%%%%%
\section{Summation formulas with the factor $k$}
%%%%%%%%%%%%%%%%%%%%%%%%%%%%%%%%%%%%%%%%%%%%%%%%%%%%%%%%%%%%%%%%%%%
Setting $y=x+1$ in \eqref{source} and considering the relation:
\[\xxqdn\xqdn\sum_{k=1}^n\frac{\binm{x+1+k}{k}}{\binm{x+k}{k}}H_k(x)=\sum_{k=1}^nH_k(x)+\frac{1}{x+1}\sum_{k=1}^nkH_k(x),\]
 we gain the following equation by using Theorem \ref{thm-a}.

\begin{prop}\label{prop-a}
For $x\in \mathbb{C}$, there holds the summation formula:
 \bnm
\sum_{k=1}^nkH_k(x)=\frac{(x+n+1)(n-x)}{2}H_n(x)+\frac{(2x-n+1)n}{4}.
 \enm
 \end{prop}

\begin{corl}[$x=p$ with $p\in\mathbb{ N}_0$ in Proposition \ref{prop-a}]
 \bnm
\,\xqdn\sum_{k=1}^nkH_{p+k}
=\frac{(n-p)(p+n+1)}{2}H_{p+n}+\frac{p(p+1)}{2}H_p-\frac{n(n-2p-1)}{4}.
 \enm
\end{corl}

\begin{thm}\label{thm-b}
For $x\in \mathbb{C}$ and $\l\in\mathbb{ N}_0$, there holds the
summation formula:
 \bnm
\qquad\sum_{k=1}^nkH_k^{\langle
\l+2\rangle}(x)=\frac{(x+n+1)(n-x)}{2}H_n^{\langle
\l+2\rangle}(x)+\frac{2x+1}{2}H_n^{\langle
\l+1\rangle}(x)-\frac{H_n^{\langle \l\rangle}(x)}{2}.
 \enm
\end{thm}

\begin{proof}
Applying the derivative operator $\mathcal{D}_x$ to Proposition
\ref{prop-a}, we achieve the case $\l=0$ of Theorem \ref{thm-b}:
 \bnm
\sum_{k=1}^nkH_k^{\langle2\rangle}(x)
=\frac{(x+n+1)(n-x)}{2}H_n^{\langle2\rangle}(x)
+\frac{2x+1}{2}H_n(x)-\frac{n}{2}.
 \enm
 Suppose that the following identity
 \bnm
\:\:\sum_{k=1}^nkH_k^{\langle \l+2\rangle}(x)
=\frac{(x+n+1)(n-x)}{2} H_n^{\langle \l+2\rangle}(x)
+\frac{2x+1}{2}H_n^{\langle \l+1\rangle}(x)
 -\frac{H_n^{\langle \l\rangle}(x)}{2}.
 \enm
  is true. Applying the derivative
operator $\mathcal{D}_x$ to the last equation, we have
 \bnm
\quad\:\:\sum_{k=1}^nkH_k^{\langle \l+3\rangle}(x)
=\frac{(x+n+1)(n-x)}{2} H_n^{\langle \l+3\rangle}(x)
+\frac{2x+1}{2}H_n^{\langle \l+2\rangle}(x)
 -\frac{H_n^{\langle \l+1\rangle}(x)}{2}.
 \enm
This proves Theorem \ref{thm-b} inductively.
\end{proof}

Taking $x=p$ in Theorem \ref{thm-b}, we attain the following
equation.

\begin{corl}
 For $\l,p\in\mathbb{ N}_0$, there holds the summation
formula:
 \bnm
\sum_{k=1}^nkH_{p+k}^{\langle
\l+2\rangle}&&\xqdn\!=\frac{(p+n+1)(n-p)}{2}H_{p+n}^{\langle
\l+2\rangle}+\frac{p(p+1)}{2}H_{p}^{\langle
\l+2\rangle}\\&&\xqdn\!+\:\frac{2p+1}{2}\big(H_{p+n}^{\langle
\l+1\rangle}-H_{p}^{\langle \l+1\rangle}\big)-\frac{H_{p+n}^{\langle
\l\rangle}-H_{p}^{\langle \l\rangle}}{2}.
 \enm
\end{corl}

%%%%%%%%%%%%%%%%%%%%%%%%%%%%%%%%%%%%%%%%%%%%%%%%%%%%%%%%%%%%%%%%%%
%%%%%%%%%%%%%%%%%%%%%%%%%%%%%%%%%%%%%%%%%%%%%%%%%%%%%%%%%%%%%%%%%%%
\section{Summation formulas with the factor $k^2$}
%%%%%%%%%%%%%%%%%%%%%%%%%%%%%%%%%%%%%%%%%%%%%%%%%%%%%%%%%%%%%%%%%%%
Letting $y=x+2$ in \eqref{source} and considering the relation:
 \bnm
 \sum_{k=1}^n\frac{\binm{x+2+k}{k}}{\binm{x+k}{k}}H_k(x)&&\xqdn\!
 =\sum_{k=1}^nH_k(x)+\frac{2x+3}{(x+1)(x+2)}\sum_{k=1}^nkH_k(x)\\&&\xqdn\!
 +\:\frac{1}{(x+1)(x+2)}\sum_{k=1}^nk^2H_k(x),
  \enm
 we get the following equation by using Theorem \ref{thm-a} and Proposition \ref{prop-a}.

\begin{prop} \label{prop-b}
 For $x\in \mathbb{C}$, there holds the summation formula:
 \bnm
\qqdn\sum_{k=1}^nk^2H_k(x)&&\xqdn\!=\frac{x(x+1)(2x+1)+n(n+1)(2n+1)}{6}H_n(x)\\
&&\xqdn\!-\:\frac{(12x^2+12x-6xn+4n^2-3n-1)n}{36}.
 \enm
\end{prop}

\begin{corl}
[$x=p$ with $p\in\mathbb{ N}_0$ in Proposition \ref{prop-b}]
 \bnm
\sum_{k=1}^nk^2H_{p+k}
&&\xqdn\!=\frac{(p+n+1)(2n^2+n-2pn+p+2p^2)}{6}H_{p+n}\\
&&\xqdn\!-\:\frac{p(p+1)(2p+1)}{6}H_p-\frac{n(4n^2-3n-6pn+12p+12p^2-1)}{36}.
 \enm
\end{corl}

Applying the derivative operator $\mathcal{D}_x$ to Proposition
\ref{prop-b}, we gain the following equation.

\begin{prop}\label{prop-c}
For $x\in \mathbb{C}$, there holds the summation formula:
 \bnm
\sum_{k=1}^nk^2H_k^{\langle2\rangle}(x)
&&\xqdn\!=\frac{x(x+1)(2x+1)+n(n+1)(2n+1)}{6}H_n^{\langle
2\rangle}(x)\\&&\xqdn\!-\:\frac{6x^2+6x+1}{6}H_n(x)+\frac{(4x+2-n)n}{6}.
 \enm
\end{prop}

\begin{corl}
[$x=p$ with $p\in\mathbb{ N}_0$ in Proposition \ref{prop-c}]
  \bnm
\sum_{k=1}^nk^2H_{p+k}^{\langle2\rangle}
&&\xqdn\!=\frac{p(p+1)(2p+1)+n(n+1)(2n+1)}{6}H_{p+n}^{\langle
2\rangle}-\frac{p(p+1)(2p+1)}{6}H_{p}^{\langle
2\rangle}\\&&\xqdn\!-\:\frac{6p^2+6p+1}{6}\big(H_{p+n}-H_p\big)+\frac{(4p+2-n)n}{6}.
 \enm
\end{corl}

\begin{thm}\label{thm-c}
For $x\in \mathbb{C}$ and $\l\in\mathbb{ N}_0$, there holds the
summation formula:
 \bnm
\sum_{k=1}^nk^2H_k^{\langle
\l+3\rangle}(x)&&\xqdn\!=\frac{x(x+1)(2x+1)+n(n+1)(2n+1)}{6}
H_n^{\langle\l+3\rangle}(x)\\&&\xqdn\!-\:\frac{6x^2+6x+1}{6}H_n^{\langle\l+2\rangle}(x)
+\frac{2x+1}{2}H_n^{\langle\l+1\rangle}(x)-\frac{H_n^{\langle\l\rangle}(x)}{3}.
 \enm
\end{thm}

\begin{proof}
Applying the derivative operator $\mathcal{D}_x$ to Proposition
\ref{prop-c}, we achieve the case $\l=0$ of Theorem \ref{thm-c}:
 \bnm
\sum_{k=1}^nk^2H_k^{\langle3\rangle}(x)&&\xqdn\!=\frac{x(x+1)(2x+1)+n(n+1)(2n+1)}{6}
H_n^{\langle3\rangle}(x)\\&&\xqdn\!-\:\frac{6x^2+6x+1}{6}H_n^{\langle2\rangle}(x)
+\frac{2x+1}{2}H_n(x)-\frac{n}{3}.
 \enm
 Suppose that the following identity
 \bnm
\sum_{k=1}^nk^2H_k^{\langle
\l+3\rangle}(x)&&\xqdn\!=\frac{x(x+1)(2x+1)+n(n+1)(2n+1)}{6}
H_n^{\langle\l+3\rangle}(x)\\&&\xqdn\!-\:\frac{6x^2+6x+1}{6}H_n^{\langle\l+2\rangle}(x)
+\frac{2x+1}{2}H_n^{\langle\l+1\rangle}(x)-\frac{H_n^{\langle\l\rangle}(x)}{3}
 \enm
is true. Applying the derivative operator $\mathcal{D}_x$ to the
last equation, we have
 \bnm
\sum_{k=1}^nk^2H_k^{\langle \l+4\rangle}(x)&&\xqdn\!
=\frac{x(x+1)(2x+1)+n(n+1)(2n+1)}{6}
 H_n^{\langle \l+4\rangle}(x)\\&&\xqdn\!
-\:\frac{6x^2+6x+1}{6}H_n^{\langle \l+3\rangle}(x)
+\frac{2x+1}{2}H_n^{\langle \l+2\rangle}(x)
 -\frac{H_n^{\langle \l+1\rangle}(x)}{3}.
 \enm
This proves Theorem \ref{thm-c} inductively.
\end{proof}

Making $x=p$ in Theorem \ref{thm-c}, we attain the following
equation.

\begin{corl}
For $\l,p\in\mathbb{ N}_0$, there holds the summation formula:
 \bnm
\sum_{k=1}^nk^2H_{p+k}^{\langle\l+3\rangle}
&&\xqdn\!=\frac{p(p+1)(2p+1)+n(n+1)(2n+1)}{6} H_{p+n}^{\langle\l+3\rangle}\\
&&\xqdn\!-\:\frac{p(p+1)(2p+1)}{6}H_p^{\langle\l+3\rangle}
 -\frac{6p^2+6p+1}{6}\big(H_{p+n}^{\langle\l+2\rangle}-H_{p}^{\langle\l+2\rangle}\big)
\\&&\xqdn\!+\:\frac{2p+1}{2}\big(H_{p+n}^{\langle\l+1\rangle}-H_{p}^{\langle\l+1\rangle}\big)
-\frac{H_{p+n}^{\langle\l\rangle}-H_{p}^{\langle\l\rangle}}{3}.
 \enm
\end{corl}

%%%%%%%%%%%%%%%%%%%%%%%%%%%%%%%%%%%%%%%%%%%%%%%%%%%%%%%%%%%%%%%%%%%
%%%%%%%%%%%%%%%%%%%%%%%%%%%%%%%%%%%%%%%%%%%%%%%%%%%%%%%%%%%%%%%%%%%
\section{Summation formulas with the factor $k^3$}
%%%%%%%%%%%%%%%%%%%%%%%%%%%%%%%%%%%%%%%%%%%%%%%%%%%%%%%%%%%%%%%%%%%
Setting $y=x+3$ in \eqref{source} and considering the relation:
 \bnm
 \xqdn\sum_{k=1}^n\frac{\binm{x+3+k}{k}}{\binm{x+k}{k}}H_k(x)&&\xqdn\!
 =\sum_{k=1}^nH_k(x)+\frac{3x^2+12x+11}{(x+1)(x+2)(x+3)}\sum_{k=1}^nkH_k(x)\\&&\xqdn\!
 +\:\frac{3}{(x+1)(x+3)}\sum_{k=1}^nk^2H_k(x)\\&&\xqdn\!
 +\:\frac{1}{(x+1)(x+2)(x+3)}\sum_{k=1}^nk^3H_k(x),
  \enm
 we get the following equation by using Theorem \ref{thm-a}, Proposition \ref{prop-a} and Proposition
\ref{prop-b}.

\begin{prop}\label{prop-d}
For $x\in \mathbb{C}$, there holds the summation formula:
 \bnm
\sum_{k=1}^nk^3H_k(x)&&\xqdn\!=\frac{(n-x)(x+n+1)(x^2+x+n+n^2)}{4}H_n(x)\\
&&\xqdn\!-\:\frac{(12x^3+18x^2-6nx^2+2x-6xn+4n^2x-2+3n+2n^2-3n^3)n}{48}.
 \enm
\end{prop}

\begin{corl} [$x=p$ with $p\in\mathbb{ N}_0$ in Proposition \ref{prop-d}]
 \bnm
\!\qdn\sum_{k=1}^nk^3H_{p+k}&&\xqdn\!=\frac{(n-p)(p+n+1)(p^2+p+n+n^2)}{4}H_{p+n}
+\frac{p^2(p+1)^2}{4}H_{p}\\
&&\xqdn\!-\:\frac{(12p^3+18p^2-6np^2+2p-6pn+4n^2p-2+3n+2n^2-3n^3)n}{48}.
 \enm
\end{corl}

Applying the derivative operator $\mathcal{D}_x$ to Proposition
\ref{prop-d}, we gain the following equation.

\begin{prop}\label{prop-e}
For $x\in \mathbb{C}$, there holds the summation formula:
 \bnm
\sum_{k=1}^nk^3H_k^{\langle2\rangle}(x)
&&\xqdn\!=\frac{(n-x)(x+n+1)(x^2+x+n+n^2)}{4}
 H_n^{\langle2\rangle}(x)\\&&\xqdn\!+\:
 \frac{x(x+1)(2x+1)}{2}H_n(x)-\frac{(18x^2+18x-6nx+1-3n+2n^2)n}{24}.
 \enm
\end{prop}

\begin{corl}[$x=p$ with $p\in\mathbb{ N}_0$ in Proposition \ref{prop-e}]
 \bnm
\sum_{k=1}^nk^3H_{p+k}^{\langle2\rangle}
&&\xqdn\!=\frac{(n-p)(p+n+1)(p^2+p+n+n^2)}{4}
H_{p+n}^{\langle2\rangle}+\frac{p^2(p+1)^2}{4}
H_{p}^{\langle2\rangle}
\\&&\xqdn\!+\:\frac{p(p+1)(2p+1)}{2}\big(H_{p+n}-H_p\big)
-\frac{(18p^2+18p-6np+1-3n+2n^2)n}{24}.
 \enm
\end{corl}

Applying the derivative operator $\mathcal{D}_x$ to Proposition
\ref{prop-e}, we achieve the following equation.

\begin{prop}\label{prop-f}
For $x\in \mathbb{C}$, there holds the summation formula:
 \bnm
\sum_{k=1}^nk^3H_k^{\langle3\rangle}(x)
&&\xqdn\!=\frac{(n-x)(x+n+1)(x^2+x+n+n^2)}{4}
H_n^{\langle3\rangle}(x)\\&&\xqdn\!+\:
\frac{x(x+1)(2x+1)}{2}H_n^{\langle2\rangle}(x)
-\frac{6x^2+6x+1}{4}H_{n}(x)+\frac{(6x+3-n)n}{8}.
 \enm
\end{prop}

\begin{corl}[$x=p$ with $p\in\mathbb{ N}_0$ in Proposition \ref{prop-f}]
 \bnm
\sum_{k=1}^nk^3H_{p+k}^{\langle3\rangle}
&&\xqdn\!=\frac{(n-p)(p+n+1)(p^2+p+n+n^2)}{4}
H_{p+n}^{\langle3\rangle}+\frac{p^2(p+1)^2}{4}
H_{p}^{\langle3\rangle}
\\&&\xqdn\!+\:\frac{p(p+1)(2p+1)}{2}
\big(H_{p+n}^{\langle2\rangle}-H_{p}^{\langle2\rangle}\big)
-\frac{6p^2+6p+1}{4}\big(H_{p+n}-H_{p}\big)
\\&&\xqdn\!+\:\frac{(6p+3-n)n}{8}.
 \enm
\end{corl}

\begin{thm}\label{thm-d}
For $x\in \mathbb{C}$ and $\l\in\mathbb{ N}_0$, there holds the
summation formula:
  \bnm
\xqdn\sum_{k=1}^nk^3H_k^{\langle \l+4\rangle}(x)
&&\xqdn\!=\frac{(n-x)(x+n+1)(x^2+x+n+n^2)}{4}
 H_n^{\langle\l+4\rangle}(x)\\&&\xqdn\!+\:
 \frac{x(x+1)(2x+1)}{2}H_n^{\langle\l+3\rangle}(x)
 -\frac{6x^2+6x+1}{4}H_{n}^{\langle\l+2\rangle}(x)\\&&\xqdn\!+\:
 \frac{2x+1}{2}H_{n}^{\langle\l+1\rangle}(x)
 -\frac{H_{n}^{\langle \l\rangle}(x)}{4}.
 \enm
\end{thm}

\begin{proof}
Applying the derivative operator $\mathcal{D}_x$ to Proposition
\ref{prop-f}, we attain the case $\l=0$ of Theorem \ref{thm-d}:
 \bnm
\xxqdn\sum_{k=1}^nk^3H_k^{\langle4\rangle}(x)
&&\xqdn\!=\frac{(n-x)(x+n+1)(x^2+x+n+n^2)}{4}
H_n^{\langle4\rangle}(x)\\&&\xqdn\!+\:
 \frac{x(x+1)(2x+1)}{2}H_n^{\langle3\rangle}(x)
 -\frac{6x^2+6x+1}{4}H_{n}^{\langle2\rangle}(x)\\&&\xqdn\!+\:
 \frac{2x+1}{2}H_{n}(x)-\frac{n}{4}.
 \enm
 Suppose that the following identity
\bnm
 \qdn\sum_{k=1}^nk^3H_k^{\langle\l+4\rangle}(x)
&&\xqdn\!=\frac{(n-x)(x+n+1)(x^2+x+n+n^2)}{4}
H_n^{\langle\l+4\rangle}(x)\\&&\xqdn\!+\:
\frac{x(x+1)(2x+1)}{2}H_n^{\langle\l+3\rangle}(x)
-\frac{6x^2+6x+1}{4}H_{n}^{\langle\l+2\rangle}(x)\\&&\xqdn\!+\:
\frac{2x+1}{2}H_{n}^{\langle\l+1\rangle}(x)
-\frac{H_{n}^{\langle\l\rangle}(x)}{4}
 \enm
is true. Applying the derivative operator $\mathcal{D}_x$ to the
last equation, we have
 \bnm
\qdn\sum_{k=1}^nk^3H_k^{\langle\l+5\rangle}(x)
&&\xqdn\!=\frac{(n-x)(x+n+1)(x^2+x+n+n^2)}{4}
H_n^{\langle\l+5\rangle}(x)\\&&\xqdn\!+\:
\frac{x(x+1)(2x+1)}{2}H_n^{\langle\l+4\rangle}(x)
-\frac{6x^2+6x+1}{4}H_{n}^{\langle\l+3\rangle}(x)\\&&\xqdn\!+\:
\frac{2x+1}{2}H_{n}^{\langle\l+2\rangle}(x)
-\frac{H_{n}^{\langle\l+1\rangle}(x)}{4}.
 \enm
This proves Theorem \ref{thm-d} inductively.
\end{proof}

Taking $x=p$ in Theorem \ref{thm-d}, we get the following equation.
\begin{corl}
 For $\l,p\in\mathbb{ N}_0$, there holds the summation formula:
 \bnm
\qdn\sum_{k=1}^nk^3H_{p+k}^{\langle \l+4\rangle}
&&\xqdn\!=\frac{(n-p)(p+n+1)(p^2+p+n+n^2)}{4}
H_{p+n}^{\langle\l+4\rangle}+\frac{p^2(p+1)^2}{4}
H_{p}^{\langle\l+4\rangle}
\\&&\xqdn\!+\:\frac{p(p+1)(2p+1)}{2}
\big(H_{p+n}^{\langle \l+3\rangle}\!-\!
H_{p}^{\langle\l+3\rangle}\big)\!-\!\frac{6p^2+6p+1}{4}
\big(H_{p+n}^{\langle\l+2\rangle}
\!-\!H_{p}^{\langle\l+2\rangle}\big)\\&&\xqdn\!+\:
\frac{2p+1}{2}\big(H_{p+n}^{\langle \l+1\rangle}
-H_{p}^{\langle\l+1\rangle}\big)
-\frac{H_{p+n}^{\langle\l\rangle}-H_{p}^{\langle \l\rangle}}{4}.
 \enm
\end{corl}

%%%%%%%%%%%%%%%%%%%%%%%%%%%%%%%%%%%%%%%%%%%%%%%%%%%%%%%%%%%%%%%%%%%
%%%%%%%%%%%%%%%%%%%%%%%%%%%%%%%%%%%%%%%%%%%%%%%%%%%%%%%%%%%%%%%%%%%
\section{Summation formulas with the factor $k^4$}
%%%%%%%%%%%%%%%%%%%%%%%%%%%%%%%%%%%%%%%%%%%%%%%%%%%%%%%%%%%%%%%%%%%
Letting $y=x+4$ in \eqref{source} and considering the relation:
 \bnm
 \xqdn\sum_{k=1}^n\frac{\binm{x+4+k}{k}}{\binm{x+k}{k}}H_k(x)&&\xqdn\!
 =\sum_{k=1}^nH_k(x)+\frac{2(2x+5)(x^2+5x+5)}{(x+1)(x+2)(x+3)(x+4)}\sum_{k=1}^nkH_k(x)\\&&\xqdn\!
 +\:\frac{6x^2+30x+35}{(x+1)(x+2)(x+3)(x+4)}\sum_{k=1}^nk^2H_k(x)\\&&\xqdn\!
 +\:\frac{4x+10}{(x+1)(x+2)(x+3)(x+4)}\sum_{k=1}^nk^3H_k(x)\\&&\xqdn\!
 +\:\frac{1}{(x+1)(x+2)(x+3)(x+4)}\sum_{k=1}^nk^4H_k(x),
  \enm
we gain the following equation by using Theorem \ref{thm-a},
Proposition \ref{prop-a}, Proposition \ref{prop-b} and Proposition
\ref{prop-d}.

\begin{prop}\label{prop-g}
For $x\in \mathbb{C}$, there holds the summation formula:
 \bnm
\qqdn\sum_{k=1}^nk^4H_k(x)&&\xqdn\!=\frac{6x^5+15x^4+10x^3-x-n+10n^3+15n^4+6n^5}{30}H_n(x)\\
&&\xqdn\!-\:\frac{(72n^4-45n^3-130n^2+75n+28)n}{1800}\\
&&\xqdn\!-\:\frac{(12x^3+24x^2+7x-6nx^2-9nx+4n^2x-5+2n+4n^2-3n^3)nx}{60}.
 \enm
\end{prop}

\begin{corl}[$x=p$ with $p\in\mathbb{ N}_0$ in Proposition \ref{prop-g}]
\bnm
\!\qqdn\sum_{k=1}^nk^4H_{p+k}&&\xqdn\!=\frac{6p^5+15p^4+10p^3-p-n+10n^3+15n^4+6n^5}{30}H_{p+n}\\
&&\xqdn\!-\:\frac{6p^5+15p^4+10p^3-p}{30}H_p-\frac{(72n^4-45n^3-130n^2+75n+28)n}{1800}\\
&&\xqdn\!-\:\frac{(12p^3+24p^2+7p-6np^2-9np+4n^2p-5+2n+4n^2-3n^3)np}{60}.
 \enm
\end{corl}

Applying the derivative operator $\mathcal{D}_x$ to Proposition
\ref{prop-g}, we achieve the following equation.

\begin{prop}\label{prop-h}
For $x\in \mathbb{C}$, there holds the summation formula:
 \bnm
\:\;\sum_{k=1}^nk^4H_k^{\langle2\rangle}(x)
&&\xqdn\!=\frac{6x^5+15x^4+10x^3-x-n+10n^3+15n^4+6n^5}{30}H_n^{\langle2\rangle}(x)
\\&&\xqdn\!-\:\frac{30x^2(x+1)^2-1}{30}H_n(x)
\\&&\xqdn\!+\:\frac{(48x^3+72x^2-18nx^2+14x-18nx+8n^2x-5+2n+4n^2-3n^3)n}{60}.
 \enm
\end{prop}

\begin{corl}[$x=p$ with $p\in\mathbb{ N}_0$ in Proposition \ref{prop-h}]
 \bnm
\:\;\sum_{k=1}^nk^4H_{p+k}^{\langle2\rangle}
&&\xqdn\!=\frac{6p^5+15p^4+10p^3-p-n+10n^3+15n^4+6n^5}{30}H_{p+n}^{\langle2\rangle}
\\&&\xqdn\!-\:\frac{6p^5+15p^4+10p^3-p}{30}H_{p}^{\langle2\rangle}-\frac{30p^2(p+1)^2-1}{30}\big(H_{p+n}-H_p\big)
\\&&\xqdn\!+\:\frac{(48p^3+72p^2-18np^2+14p-18np+8n^2p-5+2n+4n^2-3n^3)n}{60}.
 \enm
\end{corl}

Applying the derivative operator $\mathcal{D}_x$ to Proposition
\ref{prop-h}, we attain the following equation.

\begin{prop}\label{prop-i}
For $x\in \mathbb{C}$, there holds the summation formula:
 \bnm
\sum_{k=1}^nk^4H_k^{\langle3\rangle}(x)
&&\xqdn\!=\frac{6x^5+15x^4+10x^3-x-n+10n^3+15n^4+6n^5}{30}H_n^{\langle3\rangle}(x)
\\&&\xqdn\!-\:\frac{30x^2(x+1)^2-1}{30}H_n^{\langle2\rangle}(x)+x(x+1)(2x+1)H_n(x)
\\&&\xqdn\!-\:\frac{(72x^2+72x-18nx+7-9n+4n^2)n}{60}.
 \enm
\end{prop}

\begin{corl}[$x=p$ with $p\in\mathbb{ N}_0$ in Proposition \ref{prop-i}]
  \bnm
\sum_{k=1}^nk^4H_{p+k}^{\langle3\rangle}
&&\xqdn\!=\frac{6p^5+15p^4+10p^3-p-n+10n^3+15n^4+6n^5}{30}H_{p+n}^{\langle3\rangle}
\\&&\xqdn\!-\:\frac{6p^5+15p^4+10p^3-p}{30}H_{p}^{\langle3\rangle}-\frac{30p^2(p+1)^2-1}{30}
\big(H_{p+n}^{\langle2\rangle}-H_{p}^{\langle2\rangle}\big)
\\&&\xqdn\!+\:p(p+1)(2p+1)\big(H_{p+n}-H_p\big)\!-\!\frac{(72p^2+72p-18np+7-9n+4n^2)n}{60}.
 \enm
\end{corl}

Applying the derivative operator $\mathcal{D}_x$ to Proposition
\ref{prop-i}, we get the following equation.

\begin{prop}\label{prop-j}
For $x\in \mathbb{C}$, there holds the summation formula:
 \bnm
\xqdn\sum_{k=1}^nk^4H_k^{\langle4\rangle}(x)
&&\xqdn\!=\frac{6x^5+15x^4+10x^3-x-n+10n^3+15n^4+6n^5}{30}H_n^{\langle4\rangle}(x)
\\&&\xqdn\!-\:\frac{30x^2(x+1)^2-1}{30}H_n^{\langle3\rangle}(x)
+x(x+1)(2x+1)H_n^{\langle2\rangle}(x)
\\&&\xqdn\!-\:\frac{6x^2+6x+1}{3}H_n(x)+\frac{(8x+4-n)n}{10}.
 \enm
\end{prop}

\begin{corl}[$x=p$ with $p\in\mathbb{ N}_0$ in Proposition \ref{prop-j}]
  \bnm
\sum_{k=1}^nk^4H_{p+k}^{\langle4\rangle}
&&\xqdn\!=\frac{6p^5+15p^4+10p^3-p-n+10n^3+15n^4+6n^5}{30}H_{p+n}^{\langle4\rangle}
\\&&\xqdn\!-\:\frac{6p^5+15p^4+10p^3-p}{30}H_{p}^{\langle4\rangle}
 -\frac{30p^2(p+1)^2-1}{30}\big(H_{p+n}^{\langle3\rangle}-H_p^{\langle3\rangle}\big)
\\&&\xqdn\!+\:p(p+1)(2p+1)\big(H_{p+n}^{\langle2\rangle}-H_p^{\langle2\rangle}\big)
-\frac{6p^2+6p+1}{3}\big(H_{p+n}-H_p\big)
\\&&\xqdn\!+\:\frac{(8p+4-n)n}{10}.
 \enm
\end{corl}

\begin{thm}\label{thm-e}
For $x\in \mathbb{C}$ and $\l\in\mathbb{ N}_0$, there holds the
summation formula:
  \bnm
\:\sum_{k=1}^nk^4H_k^{\langle \l+5\rangle}(x)
&&\xqdn\!=\frac{6x^5+15x^4+10x^3-x-n+10n^3+15n^4+6n^5}{30}H_n^{\langle\l+5\rangle}(x)
\\&&\xqdn\!-\:\frac{30x^2(x+1)^2-1}{30}H_n^{\langle\l+4\rangle}(x)
+x(x+1)(2x+1)H_n^{\langle\l+3\rangle}(x)
\\&&\xqdn\!-\:\frac{6x^2+6x+1}{3}H_n^{\langle\l+2\rangle}(x)+\frac{2x+1}{2}H_n^{\langle\l+1\rangle}(x)
-\frac{H_n^{\langle\l\rangle}(x)}{5}.
 \enm
\end{thm}

\begin{proof}
Applying the derivative operator $\mathcal{D}_x$ to Proposition
\ref{prop-j}, we gain the case $\l=0$ of Theorem \ref{thm-e}:
  \bnm
\xqdn\sum_{k=1}^nk^4H_k^{\langle5\rangle}(x)
&&\xqdn\!=\frac{6x^5+15x^4+10x^3-x-n+10n^3+15n^4+6n^5}{30}H_n^{\langle5\rangle}(x)
\\&&\xqdn\!-\:\frac{30x^2(x+1)^2-1}{30}H_n^{\langle4\rangle}(x)
+x(x+1)(2x+1)H_n^{\langle3\rangle}(x)
\\&&\xqdn\!-\:\frac{6x^2+6x+1}{3}H_n^{\langle2\rangle}(x)+\frac{2x+1}{2}H_n(x)
-\frac{n}{5}.
 \enm
Suppose that the identity
  \bnm
\:\sum_{k=1}^nk^4H_k^{\langle \l+5\rangle}(x)
&&\xqdn\!=\frac{6x^5+15x^4+10x^3-x-n+10n^3+15n^4+6n^5}{30}H_n^{\langle\l+5\rangle}(x)
\\&&\xqdn\!-\:\frac{30x^2(x+1)^2-1}{30}H_n^{\langle\l+4\rangle}(x)
+x(x+1)(2x+1)H_n^{\langle\l+3\rangle}(x)
\\&&\xqdn\!-\:\frac{6x^2+6x+1}{3}H_n^{\langle\l+2\rangle}(x)+\frac{2x+1}{2}H_n^{\langle\l+1\rangle}(x)
-\frac{H_n^{\langle\l\rangle}(x)}{5}
 \enm
 is true. Applying the derivative operator
$\mathcal{D}_x$ to the last equation, we have
 \bnm
\sum_{k=1}^nk^4H_k^{\langle \l+6\rangle}(x)
&&\xqdn\!=\frac{6x^5+15x^4+10x^3-x-n+10n^3+15n^4+6n^5}{30}H_n^{\langle\l+6\rangle}(x)
\\&&\xqdn\!-\:\frac{30x^2(x+1)^2-1}{30}H_n^{\langle\l+5\rangle}(x)
+x(x+1)(2x+1)H_n^{\langle\l+4\rangle}(x)
\\&&\xqdn\!-\:\frac{6x^2+6x+1}{3}H_n^{\langle\l+3\rangle}(x)+\frac{2x+1}{2}H_n^{\langle\l+2\rangle}(x)
-\frac{H_n^{\langle\l+1\rangle}(x)}{5}.
 \enm
This proves Theorem \ref{thm-e} inductively.
\end{proof}

Making $x=p$ in Theorem \ref{thm-e}, we obtain the following
equation.

\begin{corl}
For $\l,p\in\mathbb{ N}_0$, there holds the summation formula:
 \bnm
\sum_{k=1}^nk^4H_{p+k}^{\langle \l+5\rangle}
&&\xqdn\!=\frac{6p^5+15p^4+10p^3-p-n+10n^3+15n^4+6n^5}{30}H_{p+n}^{\langle\l+5\rangle}
\\&&\xqdn\!-\:\frac{6p^5+15p^4+10p^3-p}{30}H_{p+n}^{\langle\l+5\rangle}
 -\frac{30p^2(p+1)^2-1}{30}\big(H_{p+n}^{\langle\l+4\rangle}-H_{p}^{\langle\l+4\rangle}\big)
\\&&\xqdn\!+\:p(p+1)(2p+1)\big(H_{p+n}^{\langle\l+3\rangle}-H_{p}^{\langle\l+3\rangle}\big)
 -\frac{6p^2+6p+1}{3}\big(H_{p+n}^{\langle\l+2\rangle}-H_{p}^{\langle\l+2\rangle}\big)
\\&&\xqdn\!+\:\frac{2p+1}{2}\big(H_{p+n}^{\langle\l+1\rangle}-H_{p}^{\langle\l+1\rangle}\big)
-\frac{H_{p+n}^{\langle\l\rangle}-H_{p}^{\langle\l\rangle}}{5}.
 \enm
\end{corl}

\textbf{Remark:} Further summation formulas with the factor $k^i$,
where $i$ is a positive integer greater than 4, can also be derived
in the same way. Considering that the resulting identities will
become more complicated, we shall not lay out them here.
%%%%%%%%%%%%%%%%%%%%%%%%%%%%%%%%%%%%%%%%%%%%%%%%%%%%%%%%%%%%%%%%%%
%%%%%%%%%%%%%%%%%%%%%%%%%%%%%%%%%%%%%%%%%%%%%%%%%%%%%%%%%%%%%%%%%%%

%%%%%%%%%%%%%%%%%%%%%%%%%%%%%%%%%%%%%%%%%%%%%%%%%%%%%%%%%%%%%%%%%%%
%%%%%%%%%%%%%%%%%%%%%%%%%%%%%%%%%%%%%%%%%%%%%%%%%%%%%%%%%%%%%%%%%%%
%%%%%%%%%%%%%%%%%%%%%%%%%%%%%%%%%%%%%%%%%%%%%%%%%%%%%%%%%%%%%%%%%%%

\end{document}